\documentclass[11pt]{amsart}
\usepackage{times}      
\usepackage{latexsym}   
\usepackage{amssymb}    
\usepackage{amsmath, color}    
\usepackage{amsthm}     
\usepackage{mathrsfs}   
\newcommand{\bb}{\mathbf{B}}

\newcommand{\bd}{\mathbf{D}}

\newcommand{\der}{\delta}

\newcommand{\1}{{\bf 1}}

\newcommand{\lp}{\left(}
\newcommand{\rp}{\right)}
\newcommand{\lc}{\left[}
\newcommand{\rc}{\right]}
\newcommand{\lcl}{\left\{}
\newcommand{\rcl}{\right\}}
\newcommand{\lln}{\left|}
\newcommand{\rrn}{\right|}

\newcommand{\ga}{\gamma}

\newcommand{\om}{\omega}
\newcommand{\oom}{\Omega}

\newcommand{\beq}{\begin{equation}}
\newcommand{\eeq}{\end{equation}}
\newcommand{\bea}{\begin{eqnarray}}
\newcommand{\eea}{\end{eqnarray}}
\newcommand{\beas}{\begin{eqnarray*}}
\newcommand{\eeas}{\end{eqnarray*}}

\def\msh{{\mathscr H}}

\def\cC{{\mathcal C}}

\def\me{{\mathbb  E}}

\def\mr{{\mathbb  R}}

\newcommand{\cb}{{\mathcal B}}
\newcommand{\cac}{{\mathcal C}}

\newcommand{\ce}{{\mathcal E}}
\newcommand{\cf}{{\mathcal F}}

\newcommand{\ch}{{\mathcal H}}

\newcommand{\cn}{{\mathcal N}}

\newcommand{\cp}{{\mathcal P}}

\newcommand{\crr}{{\mathcal R}}

\newcommand{\IP}{{\mathbb P}}
\newcommand{\IE}{{\mathbb E}}
\newcommand{\IR}{{\mathbb R}}

\newcommand{\D}{{\mathbb D}}
\newcommand{\EE}{{\mathbb E}}

\newcommand{\PP}{{\mathbb P}}

\newcommand{\R}{{\mathbb R}}

\newtheorem{theorem}{Theorem}[section]

\newtheorem{definition}[theorem]{Definition}

\newtheorem{lemma}[theorem]{Lemma}

\newtheorem{proposition}[theorem]{Proposition}
\newtheorem{assumption}[theorem]{Assumption}

\theoremstyle{remark}
\newtheorem{remark}[theorem]{Remark}
\theoremstyle{remark}
\newtheorem{example}[theorem]{Example}
\theoremstyle{remark}

\newtheorem{foo}[theorem]{Remarks}

\def\msh{{\mathscr H}}

\def\cC{{\mathcal C}}

\def\me{{\mathbb  E}}

\def\mr{{\mathbb  R}}

\numberwithin{equation}{section}

\title[]{fractal  dimensions of rough differential equations driven by fractional Brownian motions}

\author{Shuwen Lou}
\address{Dept. Mathematics, Statistics and Computer Science\\ University of Illinois at Chicago\\ Chicago, IL 60607.}
\email{slou@uic.edu}

\author{Cheng Ouyang}
\address{Dept. Mathematics, Statistics and Computer Science\\ University of Illinois at Chicago\\ Chicago, IL 60607.}
\email{couyang@math.uic.edu}

\begin{document}

\begin{abstract}
In this work we study fractal properties of rough differential equations driven by a fractional Brownian motions with Hurst parameter $H>\frac{1}{4}$.  In particular, we show that the Hausdorff dimension of the sample paths of the solution is $\min\{d,\frac{1}{H}\}$ and that the Hausdorff dimension of the level set $L_x=\{ t\in[\epsilon,1]: X_t=x\}$ is $1-dH$ with positive probability when $d<\frac{1}{H}$. 

\end{abstract}

\maketitle

\tableofcontents

\section{Introduction}

Random dynamical systems are  well established modeling tools for a variety of natural phenomena ranging from physics (fundamental and phenomenological) to chemistry and more recently to biology, economics, engineering sciences and mathematical finance. In many interesting models the lack of any regularity of the external inputs of the differential equation as functions of time is a technical difficulty that hampers their mathematical analysis. The theory of rough paths has been initially developed by T. Lyons \cite{Ly} in the 1990's to provide a framework to analyze a large class of driven differential equations and the precise relations between the driving signal and the output (that is the state, as function of time, of the controlled system). 

Rough paths theory provides a nice framework to study differential equations driven by Gaussian processes (see \cite{FV1}). In particular, using rough paths theory, we may define solutions of stochastic differential equations driven by a fractional Brownian motion. Let us then consider 
\begin{align}\label{sde-intro}
X_t=x+\int_0^t V_0(X_s)ds+\sum_{i=1}^d \int_0^t V_i(X_s)dB^i_s,
\end{align}
where $x \in \mathbb{R}^n$, $V_0,V_1,\cdots,V_d$ are bounded smooth vector fields on $\mr^n$ and $(B_t)_{t\ge 0}$ is a $d$-dimensional  fractional Brownian motion with Hurst parameter $H\in(\frac{1}{4},1)$. Existence and uniqueness of solutions to the above equation can be found, for example, in \cite{LQ}. In particular, when $H=\frac{1}{2}$, this notion of solution coincides with the solution of the corresponding Stratonovitch stochastic differential equation.  It is also clear now (cf. \cite{BH,CF,CLL,H-P}) that under H\"{o}rmander's condition the law of the solution $X_t$ has a smooth density $p_t(x,y)$ with respect to 
the Lebesgue measure on $\mathbb{R}^n$.

In this paper, we study fractal properties of the sample paths of solution $X$ to equation (\ref{sde-intro}). More specifically, we investigate the Hausdorff dimensions of sample paths $X([0,1])$ and level sets $L_x=\{t\in[0,1], X_t=x\}$ of $X$ on $[0,1]$. Our main result extends the classical result for fractional Brownian motions (see e.g. \cite{Xiao1} and \cite{Xiao}) and is summarized as follows.

\begin{assumption} \label{hyp:elliptic}
There exists a strictly positive constant $\lambda$ such that
\begin{equation*}
 v^{*} V(x) V^{*}(x) v  \geq \lambda \vert v \vert^2, \qquad \text{for all } v,x \in \R^n,
\end{equation*}
where we have set $V=(V_j^i)_{i=1,\ldots ,n; j=1,\ldots d}$. In particular, under this assumption, $n=d$ in equation (\ref{sde-intro}) and $V_1,...,V_d$ are vector fields on $\mr^d$.
\end{assumption}

\begin{theorem}
Let $X$ be the solution to equation (\ref{sde-intro}). We have almost surely
$$\dim_\mathcal{H} X([0,1])=\min\left\{d,\frac{1}{H}\right\}.$$ 
Moreover, for any given $x\in\mr^d$, 
\begin{itemize}
\item[i.] If $dH >1$, $L_x=\varnothing$ a.s.;

\item[ii.] If $dH <1$, $\dim_{\mathcal{H}}L_x=1-dH$
 with positive probability. 
\end{itemize}
\end{theorem}

Let us mention that  some relevant results regarding sample paths properties of $X$  has been studied in \cite{BNOT} under the same uniform ellipticity condition.  In the aforementioned paper the authors have shown that for any bounded Borel set $A\subset \R^d$ 
\begin{equation*}
\mathbb{P}\{X_t\  \mathrm{hits}\  A\ \mathrm{for}\  t\in[a,b]\} \quad \Longleftrightarrow  \quad \text{Cap}_\alpha(A) >0 \ ,
\end{equation*}
where $\text{Cap}_\alpha(A)$ is the $\alpha$-dimensional Newtonian capacity of $A$. In particular, this characterization of hitting probability in terms of capacity implies that $\frac{1}{H}$ is the critical dimension concerning whether or not the process $X$ hits a given point $x$ in $\mr^d$; that is, $X$ does not hit $x$  almost surely if $d>\frac{1}{H}$;  and $X$ hits $x$ with positive probability if $d<\frac{1}{H}$.   The main ingredient in obtaining these results is  upper and lower bound estimates for densities of $X$.   

Based on techniques developed in \cite{BNOT}, we are able to have some further study for density functions of $X$. In particular, we slightly improve the density estimate for the random vector $(X_s, X_t-X_s)$ to an exponential decay (see Theorem \ref{density bivariate} below)
\begin{equation*} 
p_{s,t-s}(\xi_1, \xi_2) \leq  \frac{C}{(t-s)^{dH}} \exp\left\{-\frac{|\xi_2|^{2\gamma}}{C|t-s|^{2\gamma^2}} \right\},
\end{equation*}
for $\gamma<H$.  Let us point out  that  we will prove in Theorem \ref{density increment} below that the density $p_{t-s}(\xi)$ of $X_t-X_s$ has a faster decay rate in $\xi$,
\begin{equation*} 
p_{t-s}(\xi) \leq  \frac{C}{(t-s)^{dH}} \exp\left\{-\lambda\frac{|\xi|^{(2H+1)\wedge 2}}{C(t-s)^{2H}} \right\}.
\end{equation*}
This is due to the fact that we need to perform integration by parts twice in order to obtain the density of $(X_s, X_t-X_s)$. For this purpose we have to sacrifice  the decay rate in order to have an extra order of smoothness in the Malliavin sense.

The organization of the paper is as follows. In Section 2, we give necessary preliminaries on rough paths and Malliavin calculus.  In Section 3, we show several tail and density estimates for $X$ that will be needed later in order to obtain our main result. Section 4 is then devoted to the  proof of our main result.

\section{preliminary mateiral}
For some fixed $H>\frac{1}{4}$, we consider $(\oom,\cf,\PP)$ the canonical probability space associated with the fractional
Brownian motion (in short fBm) with Hurst parameter $H$. That is,  $\oom=\cac_0([0,1])$ is the Banach space of continuous functions
vanishing at zero equipped with the supremum norm, $\cf$ is the Borel sigma-algebra and $\PP$ is the unique probability
measure on $\oom$ such that the canonical process $B=\{B_t=(B^1_t,\ldots,B^d_t), \; t\in [0,1]\}$ is a fractional Brownian motion with Hurst
parameter $H$.
In this context, let us recall that $B$ is a $d$-dimensional centered Gaussian process, whose covariance structure is induced by
\begin{align}\label{covariance}
R\left( t,s\right) :=\EE  B_s^j \, B_t^j
=\frac{1}{2}\left( s^{2H}+t^{2H}-|t-s|^{2H}\right),
\quad
s,t\in[0,1] \mbox{ and } j=1,\ldots,d.
\end{align}
In particular it can be shown, by a standard application of Kolmogorov's criterion, that $B$ admits a continuous version
whose paths are $\ga$-H\"older continuous for any $\ga<H$.

\subsection{Rough paths above fractional Brownian motions} In this section, we recall some basic results in rough paths theory and how a fractional Brownian motion is lifted to be a rough path. More details can be found in \cite{FV-bk} and \cite{LQ}.
For $N\in\mathbb{N}$, recall that the truncated algebra $T^{N}(\mathbb{R}%
^{d})$ is defined by
$$
T^{N}(\mathbb{R}^{d})=\bigoplus_{m=0}^{N}(\mathbb{R}%
^{d})^{\otimes m},
$$
with the convention $(\mathbb{R}^{d})^{\otimes
0}=\mathbb{R}$. The set $T^{N}(\mathbb{R}^{d})$ is equipped with a straightforward
vector space structure plus an multiplication $\otimes$. 
Let $\pi_{m}$ be the projection on the $m$-th tensor level. Then
$(T^{N}(\mathbb{R}^{d}),+,\otimes)$ is an associative algebra with unit
element $\mathbf{1} \in(\mathbb{R}^{d})^{\otimes0}$.

\smallskip

For $s<t$ and $m\geq2$, consider the simplex $\Delta_{st}^{m}=\{(u_{1}%
,\ldots,u_{m})\in\lbrack s,t]^{m};\,u_{1}<\cdots<u_{m}\} $, while the
simplices over $[0,1]$ will be denoted by $\Delta^{m}$. A continuous map
$\mathbf{x}:\Delta^{2}\rightarrow T^{N}(\mathbb{R}^{d})$ is called a
multiplicative functional if for $s<u<t$ one has $\mathbf{x}_{s,t}%
=\mathbf{x}_{s,u}\otimes\mathbf{x}_{u,t}$. An important example arises from
considering paths $x$ with finite variation: for $0<s<t$ we set
\begin{equation*}
\mathbf{x}_{s,t}^{m}=\sum_{1\leq i_{1},\ldots,i_{m}\leq d}\biggl( \int%
_{\Delta_{st}^{m}}dx^{i_{1}}\cdots dx^{i_{m}}\biggr) \,e_{i_{1}}\otimes
\cdots\otimes e_{i_{m}}, 
\end{equation*}
where $\{e_{1},\ldots,e_{d}\}$ denotes the canonical basis of $\mathbb{R}^{d}%
$, and then define the truncated {signature} of $x$ as
\[
S_{N}(x):\Delta^{2}\rightarrow T^{N}(\mathbb{R}^{d}),\qquad(s,t)\mapsto
S_{N}(x)_{s,t}:=1+\sum_{m=1}^{N}\mathbf{x}_{s,t}^{m}.
\]
The function $S_{N}(x)$ for a smooth function $x$ will be our typical example of multiplicative functional. Let us stress the fact that those elements take values in the strict subset
$G^{N}(\mathbb{R}^{d})\subset T^{N}(\mathbb{R}^{d})$,  called free nilpotent group of step $N$, and is equipped with the classical Carnot-Caratheodory norm which we simply denote by $|\cdot|$. For a path $\mathbf{x}\in\cC([0,1],G^{N}(\R^d))$, the $p$-variation norm of $\mathbf{x}$ is defined to be
\begin{align*}
\|\mathbf{x}\|_{p-{\rm var}; [0,1]}=\sup_{\Pi \subset [0,1]}\left(\sum_i |\mathbf{x}_{t_i}^{-1}\otimes \mathbf{x}_{t_{i+1}}|^p\right)^{1/p}
\end{align*}
where the supremum is taken over all subdivisions $\Pi$ of $[0,1]$.

\smallskip

With these notions in hand, let us briefly define what we mean by geometric rough path (we refer to \cite{FV-bk,LQ} for a complete overview): for $p\geq 1$, an element $x: [0,1]\to G^{\lfloor p \rfloor}(\mr^d)$ is said to be a geometric rough path if it is the $p$-var limit of a sequence $S_{\lfloor p \rfloor}(x^{m})$.  In particular, it is an element of the space
$$\cC^{p-{\rm var}; [0,1]}([0,1], G^{\lfloor p \rfloor}(\R^d))=\{\mathbf{x}\in \cC([0,1], G^{\lfloor p \rfloor}(\R^d)): \|\mathbf{x}\|_{p-{\rm var}; [0,1]}<\infty\}.
$$

According to the considerations above, in order to prove that a lift of a $d$-dimensional fBm as a geometric rough path exists it is sufficient to build enough iterated integrals of $B$ by a limiting procedure. Towards this aim, a lot of the information concerning $B$ is encoded in the rectangular increments of the covariance function $R$ (defined by \eqref{covariance}), which are given by
\begin{equation*}
R_{uv}^{st} \equiv \EE\lc (B_t^1-B_s^1) \, (B_v^1-B_u^1) \rc.
\end{equation*}
We then call 2-dimensional $\rho$-variation of $R$ the quantity
\begin{equation*}
V_{\rho}(R) \equiv
\sup \lcl
\lp \sum_{i,j} \lln R_{s_{i} s_{i+1}}^{t_{j}t_{j+1}} \rrn^{\rho} \rp^{1/\rho}; \, (s_i), (t_j)\in \Pi
\rcl,
\end{equation*}
where $\Pi$ stands again for the set of partitions of $[0,1]$.  It is know that (see, for example \cite{FV-bk}) if a process has a covariance function with finite $\rho$-variation for $\rho\in[1,2)$, it admits a lift to a geometric $p$-rough path for all $p>2\rho$. As a consequence, we have the following for fractional Brownian motions:
\begin{proposition}\label{prop:fbm-rough-path}
For a fractional Brownian motion with Hurst parameter $H$, we have $V_{\rho}(R)<\infty$ for all $\rho\ge\frac{1}{2H}$. Consequently, for $H>\frac{1}{4}$ the process $B$ admits a lift $\mathbf{B}$ as a geometric rough path of order $p$ for any $p>\frac{1}{H}$.
\end{proposition}

\subsection{Malliavin calculus for fractional Brownian motions} We introduce the basic framework of Malliavin calculus in this subsection.  The reader is invited to read the corresponding chapters in  \cite{Nu06} for further details. Let $\mathcal{E}$ be the space of $\mathbb{R}^d$-valued step
functions on $[0,1]$, and $\mathcal{H}$  the closure of
$\mathcal{E}$ for the scalar product:
\[
\langle (\mathbf{1}_{[0,t_1]} , \cdots ,
\mathbf{1}_{[0,t_d]}),(\mathbf{1}_{[0,s_1]} , \cdots ,
\mathbf{1}_{[0,s_d]}) \rangle_{\mathcal{H}}=\sum_{i=1}^d
R(t_i,s_i).
\]
Let $e_1,\ldots,e_d$ be the canonical basis of $\R^d$, there is an isometry $K^*_H: \ch \rightarrow  L^2([0,1])$  such that $$K^*_H(\1_{[0,t]}\, e_{i}) = \1_{[0,t]}  K_H(t,\cdot)\,e_{i},$$
where the kernel $K_H$ is given by 
\begin{align*}
K_H(t,s)&= c_H \, s^{\frac12 -H} \int_s^t (u-s)^{H-\frac 32} u^{H-\frac 12} \, du,  & H>1/2  \\
K_H(t,s)&= c_{H,1} \lp\frac{s}{t} \rp^{1/2-H} \, (t-s)^{H-1/2} 
+ c_{H,2} \, s^{1/2-H} \int_s^t (u-s)^{H-\frac 12} u^{H-\frac 32} \, du,  &H\le 1/2
\end{align*}
for some constants $c_H$, $c_{H,1},$ and $c_{H,2}$.

Let us remark that $\ch$ is the reproducing kernel Hilbert space for $B$. Denote by $\msh$ the Cameron-Martin space of $B$,
one can show that the operator $\crr:=\crr_H :\ch \rightarrow \msh$ given by
\begin{equation*}
\crr \psi := \int_0^\cdot K_H(\cdot,s) [K^*_H \psi](s)\, ds
\end{equation*}
defines an isometry between $\ch$ and $\msh$. 

A $\mathcal{F}$-measurable real
valued random variable $F$ is said to be cylindrical if it can be
written, for a given $n\ge 1$, as
\begin{equation*}
F=f\lp  B(\phi^1),\ldots,B(\phi^n)\rp=
f \Bigl( \int_0^{1} \langle \phi^1_s, dB_s \rangle ,\ldots,\int_0^{1}
\langle \phi^n_s, dB_s \rangle \Bigr)\;,
\end{equation*}
where $\phi^i \in \mathcal{H}$ and $f:\mathbb{R}^n \rightarrow
\mathbb{R}$ is a $C^{\infty}$ bounded function with bounded derivatives. The set of
cylindrical random variables is denoted $\mathcal{S}$.

The Malliavin derivative is defined as follows: for $F \in \mathcal{S}$, the derivative of $F$ is the $\mathbb{R}^d$ valued
stochastic process $(\mathbf{D}_t F )_{0 \leq t \leq 1}$ given by
\[
\mathbf{D}_t F=\sum_{i=1}^{n} \phi^i (t) \frac{\partial f}{\partial
x_i} \left( B(\phi^1),\ldots,B(\phi^n)  \right).
\]
More generally, we can introduce iterated derivatives. If $F \in
\mathcal{S}$, we set
\[
\mathbf{D}^k_{t_1,\ldots,t_k} F = \mathbf{D}_{t_1}
\ldots\mathbf{D}_{t_k} F.
\]
For any $p \geq 1$, it can be checked that the operator $\mathbf{D}^k$ is closable from
$\mathcal{S}$ into $\mathbf{L}^p(\oom;\mathcal{H}^{\otimes k})$. We denote by
$\mathbb{D}^{k,p}$ the closure of the class of
cylindrical random variables with respect to the norm
\[
\left\| F\right\| _{k,p}=\left( \mathbb{E}\left( F^{p}\right)
+\sum_{j=1}^k \mathbb{E}\left( \left\| \mathbf{D}^j F\right\|
_{\mathcal{H}^{\otimes j}}^{p}\right) \right) ^{\frac{1}{p}},
\]
and
\[
\mathbb{D}^{\infty}=\bigcap_{p \geq 1} \bigcap_{k
\geq 1} \mathbb{D}^{k,p}.
\]

\begin{definition}
Let $F=(F^1,\ldots , F^n)$ be a random vector whose components are in $\mathbb{D}^\infty$. Define the Malliavin matrix of $F$ by
$$\gamma_F=(\langle \mathbf{D}F^i, \mathbf{D}F^j\rangle_{\ch})_{1\leq i,j\leq n}.$$
Then $F$ is called  {\it non-degenerate} if $\gamma_F$ is invertible $a.s.$ and
$$(\det \gamma_F)^{-1}\in \cap_{p\geq1}L^p(\Omega).$$
\end{definition}
\noindent
It is a classical result that the law of a non-degenerate random vector $F=(F^1, \ldots , F^n)$ admits a smooth density with respect to the Lebesgue measure on $\mr^n$. Furthermore, denote by $C_p^\infty(\mr^n)$ the space of smooth functions whose derivatives together with itself have polynomial growth. The following integration by parts formula allows to get more quantitative estimates:

\begin{proposition}
Let $F=(F^1,...,F^n)$ be a non-degenerate random vector whose components are in $\D^\infty$, and $\gamma_F$ the Malliavin matrix of $F$. Let $G\in\D^\infty$ and $\varphi$ be a function in the space $C_p^\infty(\mr^n)$. Then for any multi-index $\alpha\in\{1,2,...,n\}^k, k\geq 1$, there exists an element $H_\alpha\in\D^\infty$ such that
$$\me[\partial_\alpha \varphi(F)G]=\me[\varphi(F)H_\alpha].$$
Moreover, the elements $H_\alpha$ are recursively given by
\begin{align*}
&H_{(i)}=\sum_{j=1}^{d}\delta\left(G(\gamma_F^{-1})^{ij}\mathbf{D}F^j\right)\\
&H_\alpha=H_{(\alpha_k)}(H_{(\alpha_1,..., \alpha_{k-1})}),
\end{align*}
and for $1\leq p<q<\infty$ we have
$$\|H_\alpha\|_{L^p}\leq C_{p,q}\|\gamma_F^{-1}\mathbf{D}F\|^k_{k, 2^{k-1}r}\|G\|_{k,q},$$
where $\frac{1}{p}=\frac{1}{q}+\frac{1}{r}$.
\end{proposition}

Consider process $W=\{W_t, t\in[0,1]\}$ defined by
$$W_t=B((K_H^*)^{-1}(\mathbf{1}_{[0,t]})).$$
One can show that $W$ is a Wiener process, and the process $B$ has the integral representation
\begin{align*}
B_t=\int_0^tK_H(t,s)dW_s.
\end{align*}

Based on the above representation, one can consider fractional Brownian motions and hence functionals of fractional Brownian motions as functionals of the underline Wiener process $W$. This observation allows one to perform Malliavin calculus with respect to the Wiener process $W$. We shall perform Malliavin calculus with respect to both $B$ and $W$. In order to distinguish them, the Malliavin derivatives (and corresponding Sobolev spaces respectively)  with respect to $W$ will be denoted by $D$  (and by $D^{k,p}$ respectively).
The relation between the two operators $\mathbf{D}$ and $D$  is given by the following (see e.g \cite{Nu06}).

\begin{proposition} 
Let ${D}^{1,2}$ be the Malliavin-Sobolev space corresponding to
the Wiener process $W$. Then $\D^{1,2}=(K^{*})^{-1}{D}^{1,2}$ and for
any $F\in {D}^{1,2}$ we have ${\mathrm D} F = K^{*} \bd F$
whenever both members of the relation are well defined. 
\end{proposition}



In order to estimate the bivariate density function for the random vector $(X_s, X_t)$, we will need a version of conditional integration by parts formula. For this purpose, we  choose to work on the underlying Wiener process $W$. The advantage of doing so is that projections on subspaces are easier to describe in a $L^{2}$ type setting.  Set $L^{2}_{t}\equiv L^{2}([t,1])$ and $\mathbb{E}_t=\mathbb{E}( \cdot | \mathcal{F}_t)$. For a random variable $F$ and $t\in [0,1]$, define  for $m\ge 0$, $p>0$,
\begin{equation*}
\|F\|_{m,p,t}=\left( \mathbb{E}_t\left[
F^{p}\right]
+\sum_{j=1}^m \mathbb{E}_t\left[ \left\| {D}^j F\right\|
_{(L^{2}_{t})^{\otimes j}}^{p}\right] \right) ^{\frac{1}{p}},
\end{equation*}
and
\begin{align}\label{eq:def-conditional-matrix}
{\Gamma}_{F,t}=
\lp\langle D F^i, D F^j\rangle_{L^{2}_{t}}\rp_{1\leq i,\,j\leq n}.
\end{align}
The following formula is borrowed from
\cite[Proposition 2.1.4]{Nu06}:

\begin{proposition} \label{prop:int-parts-cond-W}
Fix $k\geq 1$. Let $F,\,Z_s,\,G\in({D}^{\infty})^n$ be three
random vectors where $Z_s\in\cf_s$-measurable and 
$(\det_{{\Gamma}_{F+Z_s}})^{-1}$ has finite moments of
all orders. Let $g\in\cac_p^\infty(\mr^d)$. Then, for any
multi-index
$\alpha=(\alpha_1,\,\ldots,\,\,\alpha_k)\in\{1,\,\ldots,\,n\}^k$,
there exists a r.v. ${H}^s_\alpha(F,G)\in\cap_{p\geq
1}\cap_{m\geq 0} {D}^{m,p}$ such that
\begin{equation*}
\EE_s\lc(\partial_\alpha g)(F+Z_s) \, G\rc
=
\EE_s\lc g(F+Z_s) \, {H}_\alpha^s(F,G)|\rc,
\end{equation*}
where ${H}_\alpha^s(F,G)$ is
recursively defined by
\begin{equation*}
{H}_{(i)}^s(F,G)=\sum_{j=1}^n {\delta}_s\lp
G\lp{{\Gamma}}_{F,s}^{-1}\rp_{ij}D F^j\rp,
\quad
{H}_{\alpha}^s(F,G)={H}^s_{(\alpha_k)}(F,{H}^s_{(\alpha_1,\,\ldots,\,\alpha_{k-1})}(F,G)).
\end{equation*}
Here ${\delta}_s$ denotes the Skorohod integral with respect
to the Wiener process $W$ on the interval $[s,1]$.
Furthermore, the following norm estimates hold true:
\begin{equation*}
\Vert  H_{\alpha}^{s}(F,G) \Vert_{p,s} \leq 
c_{p,q} \Vert \Gamma_{F,s}^{-1}  \, D F\Vert^k_{k, 2^{k-1}r,s}\Vert G\Vert^k_{k, q,s},
\end{equation*}
 where $\frac1p=\frac1q+\frac1r$.
\end{proposition}

\bigskip

\section{Tail and density estimates}
Consider the following stochastic differential equation driven by a fractional Brownian motion with Hurst parameter $H>\frac{1}{4}$,
\begin{align}\label{SDE}
X_t=x+\sum_{i=1}^d\int_0^tV_i(X_s)dB^i_s+\int_0^tV_0(X_s)ds.
\end{align}
Here $V_i, i=0,1,...,d$ are $C^\infty$-bounded vector fiefs on $\mr^d$ which form a uniform elliptic system (recall Assumption \ref{hyp:elliptic}).  Proposition \ref{prop:fbm-rough-path} ensures the existence of a lift of $B$ as a geometrical rough path, which allows one to consider equation (\ref{SDE}) as a rough differential equation. Then general rough paths theory (see e.g.~\cite{FV-bk,Gu}) together with some integrability results { (\cite{CLL, FR})} shows that equation (\ref{SDE}) admits a unique finite $p$-var continuous solution $X$ in the rough paths sense, for any $p> \frac{1}{H}$. Moreover, there exists $\lambda>0$ such that
\begin{equation*}
\me\left[\exp\lambda\left(\sup_{t\in [0,1]}|X_t|^{(2H+1)\wedge 2}\right)\right]<\infty.
\end{equation*}

The above moment estimate (or rather tail estimate) extends to the increments of $X$.
\begin{proposition}\label{prop:exp-moments-rdes}Assume the uniform elliptic condition for $V_1,...,V_d$. There exist non negative constants $C$ and $c$ not depending on $s,t\in[0,1]$ and $\xi$ such that
\begin{align*}\PP\lp \sup_{s\leq u<v\leq t} |X_v-X_u| \ge \xi \rp \le C  \exp\left\{ -\frac{c \,  \xi^{(2H+1)\wedge 2}}{(t-s)^{2H}} \right\}.
\end{align*}
\end{proposition}

\begin{proof}
We borrow the idea from \cite{FR}. Consider first the case $1/4<H<1/2$. Taking up the notation of \cite{CLL} we consider $p>2\rho$ and the control
\begin{equation*}
\om_{\bb,p}(u,v)=\|\bb\|_{p-{\rm var};[u,v]}^{p}.
\end{equation*}
Then \cite[Lemma 10.7]{FV-bk} states that there exists constant $c_V$ only depends on the vector fields $V_i's$ such that
\begin{align}\label{eq:bnd-X-pvar}
\|X\|_{p-{\rm var};[u,v]}& \le c_V \lp \|\bb\|_{p-{\rm var};[u,v]} \vee \|\bb\|_{p-{\rm var};[u,v]}^{p} \rp\\
&= c_V  \lp \om_{\bb,p}(u,v)^{1/p} \vee \om_{\bb,p}(u,v) \rp.\nonumber
\end{align}
In particular, for any $t_{i}<t_{i+1}$ we have
\begin{equation}\label{eq:ineq-davies-increments}
|\der X_{t_{i}t_{i+1}}| \le 
c_V \lp \lc \om_{\bb,p}(t_{i},t_{i+1})\rc^{1/p} \vee \om_{\bb,p}(t_{i},t_{i+1}) \rp .
\end{equation}
Consider now $\alpha\ge 1$ and construct a partition of $[s,t]$ inductively in the following way: we set $t_0=s$ and
\begin{equation*}
t_{i+1}= \inf\lcl  u >t_{i} ; \, \|\bb\|^p_{p-{\rm var};[t_{i},u]} \ge \alpha \rcl.
\end{equation*}
We then set $N_{\alpha,s,t,p}=\min\{n\ge 0;\, t_n \geq t\}$. Observe that since we have taken $\alpha\ge 1$, inequality~\eqref{eq:ineq-davies-increments} becomes $$|\der X_{t_{i}t_{i+1}}| \le  c_V \,\om_{\bb,p}(t_{i},t_{i+1}) = c_V \,\alpha.$$   

Now for any fixed $u<v$ in $[s,t]$,  suppose $l$ and $m$ are such that $u\in[t_l, t_{l+1}]$ and $v\in[t_m, t_{m+1}]$. We have
\begin{align*}
|X_u-X_v|\le& \sum_{i=0}^{l} |\der X_{t_{i}t_{i+1}}|+|\der X_{t_l,u}|+|\der X_{u,t_{l+1}}| \\
 &\quad+ \sum_{i=l+1}^{m} |\der X_{t_{i}t_{i+1}}| +|\der X_{t_m,v}|+|\der X_{v,t_{m+1}}| + \sum_{i=m+1}^{N_{\alpha,s,t,p}} |\der X_{t_{i}t_{i+1}}|\\
 \le& c_V \,\alpha \, (N_{\alpha,t,p}+2)
\end{align*}
Since fractional Brownian motions have stationary increments, Theorem 6.4 in \cite{CLL} implies that there exists constants $C$ and $c$ not depending on $s,t$ and $\xi$ such that
\begin{align}\label{tail N}
\PP \lp N_{\alpha,s,t,p} +2> \xi \rp
\leq C\exp\left\{  -\frac{c \, \xi^{2H+1}}{(t-s)^{2H}}\right\}.
\end{align}
This easily yields
\begin{align*}
\PP\lp \sup_{s\leq u<v\leq t} |X_v-X_u| \ge \xi \rp
\le \PP \lp c_V \,\alpha \, (N_{\alpha,t,p} +2)> \xi \rp
\leq C \exp\left\{-\frac{c\xi^{2H+1}}{(t-s)^{2H}} \right\},
\end{align*}
which is our claim. The case $H>1/2$ is handled along the same lines, except that the coefficient $\xi^{2H+1}$ in \eqref{tail N} is replaced by $\xi^2$.
\end{proof}

The vector $X_t$ is a typical example of a smooth random variable in the Malliavin sense.  Recall that $D$ is the Malliavin derivative operator with respect to the underline Wiener process $W$ and $\Gamma_{F,t}$ is defined in (\ref{eq:def-conditional-matrix}) for a random variable $F$. The following estimate is a restatement of Proposition 5.9 in  \cite{BNOT}.

\begin{lemma} \label{estimate bivar mall} Let $\epsilon \in (0,1)$, and consider $H\in(1/4,1)$. Then there exist constants $C,r >0$ depending on $\epsilon$ such that for $ \epsilon \le s \le t \le 1$ the following holds:
 \begin{eqnarray*}
 \Vert \Gamma_{X_{t}-X_{s},s}^{-1}   \Vert^d_{d, 2^{d+2},s} 
 &\le& 
 \frac{C}{(t-s)^{2dH}} \, \mathbb{E}_s^{\frac{d}{2^{d+2}}} \!(1+ G)  \\
 \Vert D(X_{t}-X_{s}) \Vert^d_{d, 2^{d+2},s} 
 &\le& 
 C (t-s)^{dH} \, \mathbb{E}_s^{\frac{d}{2^{d+2}}}  (1+  G),
\end{eqnarray*}
where $G$ is a random variable smooth in the Malliavin sense and has finite moments to any order.
\end{lemma}

With the above lemma in hand, we are able to obtain an upper bound for the density of $X_t-X_s$.
\begin{theorem}\label{density increment}
Fix $\epsilon\in(0,1)$. Let $p_{t-s}(z)$ be the density function of  $X_t-X_s$ for $s,t\in[\epsilon,1]$. There exits a positive constant $C$ depending on $\epsilon$
        such that for all $s<t$, we have
\begin{equation*} 
p_{t-s}(z) \leq  \frac{C}{(t-s)^{dH}} \exp\left\{-\frac{|z|^{(2H+1)\wedge 2}}{C(t-s)^{2H}} \right\}.
\end{equation*}
\end{theorem}

\begin{proof}We first write
\begin{align*}
{p}_{t-s}(z)
& = \EE\lc \delta_{z}(X_t-X_s) \,  \rc\\
 & = \EE  \EE_{s}\lc \delta_{z}(X_t-X_s)  \rc.
\end{align*}
Next, we bound $M_{st}= \EE_{s}\lc \delta_{z}(X_t-X_s)  \rc $ by first using the conditional integration by parts formula in Proposition \ref{prop:int-parts-cond-W} and then Cauchy-Schwarz inequality.
\begin{align*} 
|M_{s,t}| \le  C \Vert \Gamma_{X_{t}-X_{s},s}^{-1}   \Vert^d_{d, 2^{d+2},s} \, \Vert D(X_{t}-X_{s}) \Vert^d_{d, 2^{d+2},s}  \EE_{s}^{1/2}\lc \1_{(X_t-X_s > \xi_{2})}  \rc.
\end{align*}
Thus, owing to Lemma \ref{estimate bivar mall} we obtain:
\begin{equation}\label{eq:first-bnd-bivariate}
{p}_{t-s}(z) \le  
\frac{C}{(t-s)^{dH}}   
\EE\lc \,
\mathbb{E}_s^{\frac{d}{2^{d+1}}}\! (1+ G)^2 \,
\EE_{s}^{1/2}\!\left[ \1_{(X_t-X_s > \xi_{2})} \right] \rc.
\end{equation}
Then an application of Cauchy-Schwarz inequality together with Proposition \ref{prop:exp-moments-rdes} finishes the proof.
\end{proof}

Finally we slightly improve the estimate in \cite{BNOT} for the bivariate density of $(X_s, X_t)$.

\begin{theorem}\label{density bivariate}
Fix $\epsilon\in(0,1)$ and $\gamma<H$. Let $p_{s,t}(z_1,z_2)$ be the joint density of the random vector $(X_s, X_t)$ for $s,t\in[\epsilon,1]$. There is a positive constant $C$ depending on $\epsilon$
        such that for all $s<t$, we have
\begin{equation*} 
p_{s,t}(z_1, z_2) \leq  \frac{C}{(t-s)^{dH}} \exp\left\{-\frac{|z_1-z_2|^{2\gamma}}{C|t-s|^{2\gamma^2}} \right\}.
\end{equation*}

\end{theorem}

\begin{proof} The proof is similar to the one in the above, but with a slightly more careful estimate for the tail probability of $X_t-X_s$.

First note the existence and smoothness of the density function $p_{s,t}(z_1, z_2)$  a consequence of Proposition \ref{estimate bivar mall}.
We then write
$$
p_{s,t}(z_1, z_2)=\hat{p}_{s, t-s}(z_1, z_2-z_1), \quad \text{ for } z_1, z_2 \in \R^d,
$$
where $\hat{p}_{s, t-s}(\cdot, \cdot)$ denotes the density of the random vector $(X_s, X_t-X_s)$. We now bound the function $\hat{p}_{s, t-s}$, which shall be expressed as
\begin{align*}
\hat{p}_{s, t-s}(\xi_{1},\xi_{2})
& = \EE\lc \delta_{\xi_1}(X_s) \, \delta_{\xi_2}(X_t-X_s) \,  \rc,
\quad\text{for}\quad
\xi_{1},\xi_{2}\in\R^{d},\\
 & = \EE\lc \delta_{\xi_1}(X_s) \, \EE_{s}\lc \delta_{\xi_2}(X_t-X_s)  \rc \rc.
\end{align*}
As before, we can bound  $M_{st}= \EE_{s}\lc \delta_{\xi_2}(X_t-X_s)  \rc $ as follows.
\begin{equation*}
|M_{s,t}| \le  C \Vert \Gamma_{X_{t}-X_{s},s}^{-1}   \Vert^d_{d, 2^{d+2},s} \, \Vert D(X_{t}-X_{s}) \Vert^d_{d, 2^{d+2},s}  \EE_{s}^{1/2}\lc \1_{(X_t-X_s > \xi_{2})}  \rc.
\end{equation*}
Thus, owing to Lemma \ref{estimate bivar mall} we obtain:
\begin{equation}\label{eq:first-bnd-bivariate}
\hat{p}_{s, t-s}(\xi_{1},\xi_{2}) \le  
\frac{C}{(t-s)^{nH}}   
\EE\lc \delta_{\xi_1}(X_s) \,
\mathbb{E}_s^{\frac{d}{2^{d+1}}}\! (1+ G)^2 \,
\EE_{s}^{1/2}\!\left[ \1_{(X_t-X_s > \xi_{2})} \right] \rc
\end{equation}

To proceed, recall that $\mathbf{B}$ is the lift of $B$ as a rough path. Set
\[
\cn_{\gamma,q}(\mathbf{B})=\int_{0}^{1}\int_{0}^{1}
\frac{d(\mathbf{B}_v,\mathbf{B}_u)^{q}}{|v-u|^{\gamma q}}\,dudv,
\]
where $\gamma <H$ and $q>0$. The reason we need to consider these Besov norms $\cn_{\gamma,q}(\mathbf{B})$ is that they are smooth in the Malliavin sense.
From the Besov-H\"older embedding we have for any $\epsilon>0$ such that $\gamma+\epsilon<H$, and large enough $q$ 
\begin{align*}
\|\mathbf{B}\|_{0,1;\gamma} \le C\cn_{\gamma+\epsilon,q}(\mathbf{B})^{1/q}.
\end{align*}
Hence, it is readily checked by (\ref{eq:bnd-X-pvar}) that (note that $0<t-s<1$ in our case)
\begin{equation*}
|X_t-X_s| \le C\, |t-s|^\gamma 
(1+\cn_{\ga+\epsilon,q}(\mathbf{B})^{1/q})^{1/\gamma}.
\end{equation*}
Furthermore, {by \cite{FV1}}, for any $\gamma<H$ there exists a constant $\lambda>0$ such that 
$$\me \exp\left\{\lambda \|\mathbf{B}\|_{0,1,\gamma}^2 \right\}<\infty.$$
This together with
$$\cn_{\ga+\epsilon,q}(\mathbf{B})^{1/q}\leq \|\mathbf{B}\|_{0,1,\gamma+\epsilon},$$
implies that $\cn_{\gamma+\epsilon,q}(\mathbf{B})^{1/q}$ has Gaussian tail.
Thus, for large enough $q$ and small enough $\lambda$, we have for some constant $C>0$
\begin{equation*}
\EE_{s}\left[ \1_{(X_t-X_s > \xi_{2})} \right] \le 
 C  \exp\left\{-\lambda\frac{|\xi|^{2\gamma}}{C|t-s|^{2\gamma^2}} \right\}\EE_{s}\exp\left\{ \lambda (1+\cn_{\ga+\epsilon,q}(\mathbf{B}))^{2/q}\right\} .
\end{equation*}
Plugging this inequality into \eqref{eq:first-bnd-bivariate}, we end up with:
\begin{equation}\label{eq:bnd-density-conditional}
\hat{p}_{s, t-s}(\xi_{1},\xi_{2}) \le  \frac{C}{(t-s)^{dH}} \exp\left\{-\lambda\frac{|\xi|^{2\gamma}}{C|t-s|^{2\gamma^2}} \right\}  \,
\EE\left[ \delta_{\xi_1}(X_s) \, \Psi_1 \Psi_2  \right]
\end{equation}
where $\Psi_1$ and $\Psi_2$ are two random variables which are smooth in the Malliavin calculus sense. Also note that we may need to choose $\lambda$ even smaller to make sure Malliavin derivatives of $\Phi_2$ have moments to certain large order.  Based on the above consideration,  we can now integrate \eqref{eq:bnd-density-conditional} safely by parts in order to regularize the term $\delta_{\xi_1}(X_s)$, which finishes the proof.

\end{proof}

To close the discussion in this section, let us state a mild lower bound (strict positivity) for the density of $X_t$. We direct interested reader to \cite{BNOT} for its proof. 
\begin{theorem}\label{thm:strict-positivity-intro}
Consider the solution $X$ to equation \eqref{SDE} driven by a $d$-dimensional fractional Brownian motion with Hurst parameter $H>\frac{1}{4}$.  For each fixed $x\in\mr^d$, denote by $p_t(x,\cdot):\R^{d}\to\R_{+}$ the density function of the random variable $X_t$. Assume that Assumption \ref{hyp:elliptic} is satisfied.  Then  $p_t(x, y)>0$ for all $y\in\R^d$.
\end{theorem}

\section{Fractal dimensions of SDEs driven by fractional Brownian motions}
 We first briefly recall the definition of capacity and packing dimension, as well as their connection to Hausdorff dimension. A kernel $\kappa$ is a measurable function $\kappa: \mr^d\times\mr^d\to [0,\infty]$. For a Borel measure $\mu$ on $\mr^d$, the energy of $\mu$ with respect to the kernel $\kappa$ is defined by
 \begin{align*}
 \ce_\kappa(\mu)=\int_{\mr^d}\int_{\mr^d}\kappa(x,y)\mu(dx)\mu(dy).
 \end{align*}
 For any Borel set $\me\subset\mr^d$, the capacity of $E$ with respect to $\kappa$, denoted by $\cac_\kappa(E)$, is defined by
 $$\cac_\kappa(E)=\left[\inf_{\mu\in\cp(E)}\ce_\kappa(\mu)\right]^{-1}.$$
 Here $\cp(E)$ is the family of probability measures carried by $E$. 
 Note that $\cac_\kappa(E)>0$ if and only if there is a probability measure $\mu$ on $E$ with finite $\kappa$-energy.  Throughout our discussion, we will mostly consider the case when $\kappa(x,y)=f(|x-y|),$ where 
 \begin{align*}
 f(r)=\left\{\begin{array}{ll}
 r^{-\alpha},\quad&\mathrm{if}\ \alpha>0;\\
 \log\left(\frac{e}{r\wedge1}\right),\quad &\mathrm{if}\ \alpha=0.
 \end{array}\right.
 \end{align*}
 The corresponding energy and capacity will be denoted by $\ce_\alpha(\mu)$ and $\cac_\alpha(E)$, respectively; and the former will be called the $\alpha$-energy of $\mu$ and the latter will be called the Bessel-Riesz capacity of $E$ of order $\alpha$. The capacity dimension of $E$ is defined by 
$$\dim_c(E)=\sup\{\alpha>0; \cac_\alpha(E)>0\}.$$
It is know by Frostman's theorem (cf. \cite{Kahane:85} or \cite{Khoshnevisan}) that
$$\dim_{\ch}E=\dim_c(E),$$
for every compact subset $E$ of $\mr^d$. Hence, in order to show $\dim_\ch E\geq \alpha$ one only needs to find a measure $\mu$ on $E$ such that the $\alpha$-energy of $\mu$ is finite.

Packing dimension and packing measure were introduced as dual concept to Hausdorff dimension and Hausdorff measure. Denote by $\dim_\cp$ by packing dimension. For any $\varepsilon>0$ and any bounded set $F\subset\mr^d$, let $N(F,\varepsilon)$ be the smallest number of balls of radius $\varepsilon$ (in Euclidean metric) needed to cover $F$. Then the upper box-counting dimension of $F$ is 
\begin{align*}
\overline{\dim}_\cb F=\limsup_{\varepsilon\downarrow0}\frac{\log N(F,\varepsilon)}{-\log\varepsilon}.
\end{align*}
The packing dimension of $F$ can be defined by
\begin{align*}
\dim_\cp F=\inf\left\{\sup_n\overline{\dim}_\cb F_n: f\subset\cup_{n=1}^\infty F_n\right\}.
\end{align*}
It is known that for any bounded set $F\subset\mr^d$ (cf. \cite{Xiao}),
\begin{align}\label{dim relation}
\dim_\ch F\leq \dim_\cp F\leq \overline{\dim}_\cb F\leq d.
\end{align}

\subsection{Hausdorff dimension for SDE driven by fBm}
Recall that $X$ is the solution to
\begin{align}\label{SDE2}
X_t=x+\int_0^tV_0(X_s)ds+\sum_{i=1}^d\int_0^tV_i(X_s)dB^i_s,
\end{align}
where $V_i, i=0,1,...,d$ are $C^\infty$-bounded vector fiefs on $\mr^d$ satisfying Assumption \ref{hyp:elliptic}. Denote by$$X([0,1])=\{X(t): t\in[0,1]\}$$ the collection of sample paths of $X$ on time interval $[0,1]$. 
\begin{theorem}\label{upper bound}
We have
$$\overline{\dim}_\cb X([0,1])\leq \frac{1}{H},\quad\mathrm{a.s.}$$
\end{theorem}
\begin{proof}
For any constant $0<\gamma<H$, we have
\begin{align}\label{holder}\sup_{s,t\in[0,1]}\frac{|X_t-X_s|}{|s-t|^\gamma}=\|X\|_{\gamma;[0,1]},\end{align}
and $\|X\|_{\gamma;[0,1]}$ has moments to any order. We fix a sample path $w$ and suppress it. For any integer $n\geq2$, we divide $[0,1]$ into $m_n$ sub-intervals $\{R_{n,i}\}$ with length $n^{-1/H}$. Then
$$m_n\leq c_1n^{\frac{1}{H}}$$
and $X([0,1])$ can be covered by $X(R_{n,i}), 1\leq i\leq m_n.$ By (\ref{holder}), we see that the diameter of the image $X(R_{n,i})$ is controlled by
$$\mathrm{diam} X(R_{n,i})\leq c_2n^{-1+\delta},$$
where $\delta=(H-\gamma)/H$. Consequently, for $\varepsilon_n=c_2n^{-1+\delta}$, $X([0,1])$ can be covered by at most $m_n$ balls in $\mr^d$ of radius $\varepsilon_n$. That is 
$$N(X([0,1], \varepsilon_n)\leq c_1n^{\frac{1}{H}}.$$
This implies
$$\overline{\dim}_\cb X([0,1])\leq\frac{1}{1-\delta}\frac{1}{H},\quad \mathrm{a.s.}$$
Letting $\gamma\uparrow H$ finishes the proof.

\end{proof}

Next, we turn to the lower bound.

\begin{theorem}\label{lower bound}
Let $X$ be the solution to equation (\ref{SDE}).  We have
$$\dim_\ch X([0,1])\geq\min\left\{d,\frac{1}{H}\right\}.$$
\end{theorem}



\begin{proof}  Note that for any $\varepsilon>0$, we have $\dim_\ch X([0,1])\geq\dim_\ch X([\varepsilon,1])$. To prove the claimed result, it suffices to show that
$$\dim_\ch X([\varepsilon,1])\geq \gamma\quad \mathrm{a.s.}$$
for every $0<\gamma<\min\{d,1/H\}$.

Let $\mu_X$ be the image of measure of the Lebesgue measure on $[\varepsilon, 1]$ under the mapping $t\mapsto X_t$. Then the energy of $\mu_X$ of order $\gamma$ can be written as
$$\int_{\mr^n}\int_{\mr^n}\frac{\mu_X(dx)\mu_X(dy)}{|x-y|^\gamma}=\int_{[\varepsilon,1]}\int_{[\varepsilon,1]}\frac{dsdt}{|X_t-X_s|^\gamma}.$$
Hence, by Frostman's theorem, it is sufficient to show that for every $0<\gamma<\min\{d,1/H\}$,
\begin{align}\label{energy est}\ce_\gamma=\int_{[\varepsilon,1]}\int_{[\varepsilon,1]}\me\left(\frac{1}{|X_t-X_s|^\gamma}\right)dsdt<\infty.\end{align}
By Theorem \ref{density increment},
\begin{align*}
\me\frac{1}{|X_t-X_s|^\gamma}&\leq \int_{\mr^d}\frac{1}{|z|^\gamma}\frac{C}{|t-s|^{dH}}e^{-\frac{|z|^{(2H+1)\wedge2}}{C|t-s|^{2H}}}dz\\
&=\frac{C}{|t-s|^{\gamma H}}\int_{\mr^d}\frac{1}{|\xi|^\gamma} \exp\left\{-\frac{|\xi|^{(2H+1)\wedge2}}{|t-s|^{C[2H-(2H+1)\wedge2]}}\right\}d\xi\\
&\leq \frac{C'}{|t-s|^{\gamma H}}.
\end{align*}
In the above, we have used the fact that $\gamma<d$ to ensure the integration near $0$ is finite. We have also used that  $2H-(2H+1)\wedge2>0$ to ensure the integration at infinity is finite and uniform in $s, t\in[0,1]$. Now plug the above into (\ref{energy est}) and note that we assume $\gamma<\frac{1}{H}$. The proof is thus completed.
\end{proof}

Theorem \ref{upper bound} and Theorem \ref{lower bound} together with (\ref{dim relation}) give us the  Hausdorff dimension of sample path of $X$.
\begin{theorem}
Let $X$ be the solution to equation (\ref{SDE2}). We have almost surely
$$\dim_\mathcal{H} X([0,1])=\min\left\{d,\frac{1}{H}\right\}.$$
\end{theorem}

It is expected that similar argument also allows us to analyze the Hausdorff dimension of
$\mathrm{Gr}X([0,1])=\{(t,X(t)): t\in[0,1]\},$
the graph of $X$ on $[0,1]$.  The result is summarized in the following theorem.

\begin{theorem}
The following holds regarding the dimension of the graph of $X$.
$$\dim \mathrm{Gr} X\big([0,1]\big)=\min\left\{(1-H)d+1, \frac{1}{H}\right\}.$$
i.e., 
\begin{equation*}
\dim \mathrm{Gr}\big(X[0,1]\big)=\left\{
  \begin{array}{ll}
    \frac{1}{H}, 
&
\hbox{$Hd>1$;}
\\ \\
  (1-H)d+1,
& \hbox{$Hd<1$}

  \end{array}
\right.
\end{equation*}
\end{theorem}

\begin{proof}
We first prove the upper bound of $\dim \mathrm{Gr}\big(X[0,1]\big)$.
Given any $\delta>0$, we can cover $\mathrm{Gr} X\big([0,1]\big)$ by $m_n$ many balls in $\IR^{1+d}$ with radius $n^{-1+\delta}$, where $m_n\le cn^{1/H}$, which provides an upper bound $1/H$ for $\dim \mathrm{Gr}\big(X[0,1]\big)$. To show the other upper bound $(1-H)d+1$, we notice that as in the proof to Theorem \ref{upper bound}, each $R_{n,i}\times X_{R_{n,i}}$ can be covered by $l_n$ balls with radius $n^{-1/H}$, where
\begin{equation*}
l_n=n^{\left(\frac{1}{H}+(-1+\delta)\right)d+1+\frac{1}{H}}.
\end{equation*}
Therefore $\mathrm{Gr}X([0,1])$ can be covered by $m_n\times l_n$ balls with radius $n^{-1/H}$. Letting $\delta>0$ go to zero proves that 
\begin{equation*}
\dim \mathrm{Gr}X([0,1])\le H\left[\left(\frac{1}{H}-1\right)d+\frac{1}{H}\right]=1+(1-H)d.
\end{equation*}

Now we need to prove the lower bound for $\dim \mathrm{Gr}X([0,1])$. Since $\dim \mathrm{Gr}X\ge \dim_\mathcal{H}X$ always holds, when $(1-H)d+1>1/H$, or equivalently, $Hd>1$, $\dim \mathrm{Gr}X\ge \dim_\mathcal{H}X=1/H$. Therefore we only need to worry about the case when  $(1-H)d+1<1/H$, or equivalently, $Hd<1$. For $\gamma$ such that $d<1/H<\gamma<1+(1-H)d$ arbitrarily fixed, we will prove $\dim \mathrm{Gr}X([0,1])$ a.s.. By Frostman's theorem, we will show
\begin{equation}\label{1210630}
\int_{t=0}^1\int_{s=0}^1\IE\left[\frac{1}{\left(|t-s|^2+|X_t-X_s|^2\right)^{\gamma/2}}\right]dsdt<\infty.
\end{equation}
By the upper bound density estimate for $X_t-X_s$, we have for the left hand side of \eqref{1210630} that
\begin{align*}
&\quad \int_{t=0}^1\int_{s=0}^1\IE\left[\frac{1}{\left(|t-s|^2+|X_t-X_s|^2\right)^{\gamma/2}}\right]dsdt
\\
&\le \int_{t=0}^1\int_{s=0}^1\int_{X_t-X_s:=y\in \IR^d} \frac{1}{|t-s|^{dH}}e^{-\frac{|y|^{(2H+1)\wedge 2}}{|t-s|^{2H}}}\left(\frac{1}{(|t-s|^2+|y|^2)}\right)^{\gamma/2}dydsdt
\\
&\le \int_{t\in [0,1]}\int_{s\in [0,1]}\int_{y\in \IR^d}\frac{1}{|t-s|^{dH+\gamma}}e^{-\frac{|y|^{(2H+1)\wedge 2}}{|t-s|^{2H}}}dydsdt
\\
&\le \int_{t\in [0,1]}\int_{s\in [0,1]}\frac{1}{|t-s|^{1-\delta}}\int_{r\in [0,1]}\frac{r^{d-1}}{|t-s|^d}e^{-\frac{r^{(2H+1)\wedge 2}}{|t-s|^{2H}}}drdsdt
\\
&\asymp \int_{t\in [0,1]}\int_{s\in [0,1]}\frac{1}{|t-s|^{1-\delta}}\int_{r\in [0,1]}e^{-\frac{r^{(2H+1)\wedge 2}}{|t-s|^{2H}}}d\left(\frac{r}{|t-s|}\right)^ddsdt<\infty
\end{align*}
for any $\delta>0$ such that $\gamma+dH<1-\delta+d$. The integral converges since $2H<(2H+1)\wedge 2$. Therefore it has been shown that when $(1-H)d+1<1/H$, i.e, when $Hd<1$, $\dim \mathrm{Gr}X([0,1])\ge (1-H)d+1$. This proves that
\begin{equation*}
\dim \mathrm{Gr}X([0,1])\ge \min\left\{(1-H)d+1, \frac{1}{H}\right\}.
\end{equation*}
\end{proof}

\subsection{Hausdorff dimension of level sets}

In this section, we study Hausdorff dimension of the level sets of processes driven by fBM. Before we state the main results, we have the following two lemmas regarding the upper and lower bounds of Hausdorff dimensions of the level sets which actually matches each other. The following lemma addresses the upper bound. 
\begin{lemma}\label{124935}
Let $X=\{X(t), t\in \IR\}$ be a process driven by fBM on $t\in [\epsilon, 1]$.
\begin{equation}\label{124730}
\overline{\dim}_{\mathcal{B}}L_x\le 1-dH, \quad a.s. 
\end{equation}
and $L_x=\varnothing$ when the right hand side of \eqref{124730} is negative. 
\end{lemma}
\begin{proof}
For a fixed integer $n\ge 1$, we divide the interval $[0, 1]$ into $m_n$ subintervals $I_i$, $1\le i\le m_n$ with lengths $n^{-1/H}$. $m_n\le cn^{1/H}$ for some $c>0$. Let $0<\delta<1$ be fixed, and let $t_i$ be the left endpoint of $I_i$. In the following proof, $c$ is always some strictly positive constant which may change from line to line. We have for $1\le i\le m_n$
\begin{align}\label{930307}
\IP\left( x\in X(I_i)\right)&\le \IP\left(\max_{s,t\in I_i}\left|X_s-X_t\right|\le n^{-(1-\delta)}; x\in X(I_i)\right)+\IP\left(\max_{s,t\in I_i}\left|X_s-X_t\right|>n^{-(1-\delta)}\right).
\end{align}
We first claim that the first term on the right hand side of \eqref{930307} is bounded by $cn^{-(1-\delta)}$. Assuming the process is started at $z$ and applying Theorem \ref{density bivariate}, we have for $1\le i\le m_n$, 
\begin{align}
\IP\left(\max_{s,t\in I_i}\left|X_s-X_t\right|\le n^{-(1-\delta)}; x\in X(I_i)\right)&\le \IP\left(\left| X_{t_i}-x\right|\le n^{-(1-\delta)}\right) \nonumber
\\
&\le \int_{y\in B(x, n^{-(1-\delta)})} \frac{c}{t_i^{dH}}e^{-\frac{|y-z|^{(2H+1)\wedge 2}}{ct_i^{2H}}}dy \nonumber
\\
&\le \int_{y\in B(x, n^{-(1-\delta)})} \frac{c}{t_i^{dH}}e^{-\frac{|y-x|^{(2H+1)\wedge 2}}{ct_i^{2H}}}dy \nonumber
\\
&\le c\int_{r=0}^{n^{-(1-\delta)}}\frac{r^{d-1}}{t_i^{dH}}e^{-\frac{r^{2\wedge (2H+1)}}{ct_i^{2H}}}dr \le cn^{-(1-\delta)d}, \nonumber
\end{align}
where the last inequality is due to the fact that $t_i$ is bounded away from zero by at least $\epsilon$.

For the second term on the right hand side of \eqref{930307}, we would like to bound it from above by $e^{-cn^{2\delta}}$, which will then bound the left hand side of \eqref{930307} by $cn^{-(1-\delta)d}$. Indeed, it has been shown 
\begin{equation}\label{930353}
\IP\left(\sup_{s\in [0,t]}\left|X^x_t-x\right|\ge \xi\right)\le c\exp\left[-\frac{c_H\xi^{(2H+1)\wedge 2}}{t^{2H}}\right], \quad \text{for }t\in [0,1].
\end{equation}
It follows from \eqref{930353} that for $1\le i\le m_n$,
\begin{align*}
\IP\left[\max_{s,t\in I_i}\left|X_s-X_t\right|>n^{-(1-\delta)}\right]&\le ce^{-\frac{cn^{-(1-\delta)[(2H+1)\wedge 2]}}{n^{-\frac{1}{H}(2H)}}} \le e^{-c\left(\frac{n^{-2(1-\delta)}}{n^{-2}}\right)} =e^{-cn^{2\delta}}.
\end{align*}
Therefore, we have shown that
\begin{align}\label{102317}
\nonumber \IP\left( x\in X(I_i)\right)&\le \IP\left(\max_{s,t\in I_i}\left|X_s-X_t\right|\le n^{-(1-\delta)}; x\in X(I_i)\right)+\IP\left(\max_{s,t\in I_i}\left|X_s-X_t\right|>n^{-(1-\delta)}\right)
\\
&\le cn^{-(1-\delta)d}+e^{-cn^{2\delta}}\le cn^{-(1-\delta)d}, \quad 1\le i\le m_n
\end{align}
for sufficiently large $n$.

If $dH>1$, we choose $\delta>0$ such that $(1-\delta)d>1/H$. We denote by $N_n$ the number of times that $X(I_i)$ visits $x$, as $i$ increases from $1$ to $m_n$. It now follows from \eqref{102317} and the fact that $m_n\le cn^{1/H}$ 
\begin{equation}\label{102321}
\IE(N_n)\le cn^{1/H}\cdot n^{-(1-\delta)d}\rightarrow 0, \quad \textrm{as }n\rightarrow \infty,
\end{equation}
which combined with Fatou's lemma yields that $N_n=0$ for infinitely many $n$, since the random variables $N_n$ are integer-valued. Therefore $L_x=\varnothing$ a.s..

If $dH<1$, a covering $\displaystyle{\left\{\widetilde{I}_i\right\}_{i=1}^{m_n}}$ of $L_x$ can be defined by setting $\widetilde{I}_i=I_i$ if $x\in X(I_i)$, and $\widetilde{I}_i=\varnothing $ otherwise. Each interval $I_i$ has length $n^{-1/H}$. Let the number of such nonempty $\widetilde{I}_i$ be denoted by $M_n$. In the same way as getting \eqref{102321}, again by \eqref{102317} we have
\begin{align}\label{103343}
\IE(M_n) \le cn^{-1/H} n^{-(1-\delta)d}=cn^{1/H-(1-\delta)d}.
\end{align}
By picking $\delta>0$ sufficiently small, we may define a constant $\eta$ by 
\begin{equation*}
\eta=1/H-(1-2\delta)d>0.
\end{equation*}
We consider the sequence of integers $n_j=2^j, \; j\ge 1$. \eqref{103343} and Chebyshev's inequality imply 
\begin{equation*}
\IP\left(M_{n_j}\ge \epsilon \right)\le \frac{c}{\epsilon}\; 2^{j\left(\frac{1}{H}-(1-\delta)d\right)}, \quad \forall \epsilon>0.
\end{equation*}
Borell-Cantelli lemma yields almost surely $M_{n_j}\le cn_j^\eta$ for all $j$ large enough. This implies that $\overline{\dim}_\mathcal{B}L_x\le H\eta$ almost surely. Letting $\delta\downarrow 0$ along rational numbers proves
\begin{equation*}
\overline{\dim}_{\mathcal{B}}L_x\le 1-Hd\quad a.s.,
\end{equation*}
which proves the upper bound.
\end{proof}

The next lemma is on the lower bound of Hausdorff dimension of $L_x$. 
\begin{lemma}\label{124936}
There exists some $C>0$ such that
\begin{equation*}
\IP\{\dim_{\mathcal{H}}L_x\ge 1-Hd\}\ge C.
\end{equation*}
\end{lemma}
\begin{proof}
To prove the lower bound for $\dim_{\mathcal{H}}L_x$, we fix a small constant $\delta>0$ such that 
\begin{equation*}
\gamma :=1-(1+\delta)Hd>0.
\end{equation*}
Note that if we can prove that there is a constant $c>0$ independent of $\delta$ such that
\begin{equation}\label{103313}
\IP\left\{ \dim_{\mathcal{H}}L_x\ge \gamma \right\}\ge c,
\end{equation}
then the lower bound will follow by letting $\delta\downarrow 0$. The strategy to prove \eqref{103313} is standard and based on the capacity argument by Kahane \cite{Kahane:85}. We spell out all the details for the convenience of the readers. 

Let $\mathcal{M}_\gamma^+$ be the space of all non-negative measures on $\IR$ with finite $\gamma$-energy. It is known that $\mathcal{M}_\gamma^+$ is a complete metric space under the metric $\|\cdot \|_\gamma$ given by 
\begin{equation*}
\|\mu\|_\gamma^2 =\int_{\IR}\int_{\IR}\frac{\mu(dt)\mu(ds)}{|t-s|^\gamma}. 
\end{equation*}
We define a sequence of random positive measures $\mu_n:=\mu_n(x, \cdot)$ on the Borel sets $C$ of $[\epsilon, 1]$ by
\begin{equation}\label{103355}
\mu_n(C)=\int_C (2\pi n)^{d/2} \exp\left(-\frac{n|X_t-x|^2}{2}\right)dt. 
\end{equation}
It follows from Kahane \cite{Kahane:85} or Testard \cite{Testard} that if there exist positive constants $c_i$,$1\le i\le 3$ such that
\begin{equation}\label{103410}
\IE\left(\|\mu_n\|\right)\ge c_1, \quad \IE\left(\|\mu_n\|^2\right)\le c_2, \quad \IE\left(\|\mu_n\|_\gamma\right)\le c_3 \quad \text{ for all }n., 
\end{equation}
where $\|\mu_n\|=\mu_n([\epsilon, 1])$ denotes the total mass of $\mu_n$, then there is a subsequence of $\{\,u_n\}$, say $\{\mu_{n_k}\}$, such that $\mu_{n_k}\rightarrow \mu\in \mathcal{M}_\gamma^+$, and $\mu$ is strictly positive with probability at least $c_1^2/(2c_2)$. It then follows from \eqref{103355} that $\mu$ has its support in $L_x$ almost surely. Moreover, the third inequality of \eqref{103355} together with monotone convergence theorem imply that the $\gamma-$energy of $\mu$ is finite. Therefore Frostman's theorem yields \eqref{103313} with $c=c_1^2/(2c_2). $ It remains to verify the three inequalities in \eqref{103410}. 

In the following computation, the strictly positive constant $c$ may change from line to line. 
\begin{align}
\IE\left(\|\mu_n\|\right)&=\IE\int_{[\epsilon, 1]}(2\pi n)^{d/2} \exp \left[-\frac{n\left|X_t-x\right|^2}{2}\right]dt \nonumber 
\\ 
&=\int_{[\epsilon, 1]}(2\pi n)^{d/2}\IE\exp \left[-\frac{n\left|X_t-x\right|^2}{2}\right]dt \nonumber
\\
&\ge \int_{[\epsilon, 1]} cn^{d/2}\int_{y\in \IR^d}e^{-\frac{n|y-x|^2}{2}}p(t,x,y)dydt \nonumber 
\\
&\ge \int_{[\epsilon, 1]} cn^{d/2}\int_{|y-x|<1/n}e^{-\frac{n|y-x|^2}{2}}p(t,x,y)dydt \nonumber 
\\
&\ge \int_{[\epsilon, 1]} cn^{d/2}\int_{ |z-x|<1}p(t,x,\frac{z-x}{\sqrt{n}}+x)\frac{1}{(\sqrt{n})^d}dzdt >c>0. \nonumber
\end{align}
In the last line above we have performed a change of variable $\sqrt{n}(y-x)=z-x$ and restricted the integration over $z$ in a ball $|z-x|<1$. Then the desired low bound clearly follows by the fact that $p(t,x,y)$ is globally strictly positive (Theorem \ref{thm:strict-positivity-intro}).

For the second inequality of \eqref{103410}, by applying the transition density estimate in Theorem \ref{density bivariate}, we get
\begin{align}
&\,\IE\left(\|\mu_n\|^2\right)=\IE \left[ \int_{t\in [\epsilon, 1]}(2\pi n)^{d/2}\exp \left(-\frac{n\left|X_t-x\right|^2}{2}\right)dt \int_{s\in [\epsilon, 1]}(2\pi n)^{d/2} \exp \left(-\frac{n \left|X_s-x\right|^2}{2}\right)ds\right] \nonumber 
\\
&=2\cdot (2\pi n)^d \cdot \IE\left[\int_{t=\epsilon}^1\int_{s=\epsilon}^t\exp \left(-\frac{n\left(\left|X_t-x\right|^2+\left|X_s-x\right|^2\right)}{2}\right)^2dsdt\right] \nonumber
\\
&\le c n^d \int_{z\in \IR^d}\int_{y\in \IR^d}\int_{t=\epsilon}^1\int_{s=\epsilon}^1\exp\left(-\frac{n\left(\left|y-x\right|^2+\left|z-x\right|^2\right)}{2}\right)\frac{1}{|t-s|^{Hd}}dsdtdydz \nonumber
\\
&\le cn^d \int_{z\in \IR^d}\int_{y\in \IR^d}\int_{t=\epsilon}^1\int_{s=\epsilon}^1\frac{1}{|t-s|^{Hd}}\exp\left(-\frac{n\left(\left|y-x\right|^2+\left|z-x\right|^2\right)}{2}\right)dsdtdydz \nonumber
\\
&=cn^d \int_{t\in [\epsilon,1]}\int_{s\in [\epsilon,1]}\frac{1}{|t-s|^{Hd}}dsdt\int_{z\in \IR^d}e^{-\frac{n|x-z|^2}{4}}dz\int_{y\in \IR^d}e^{-\frac{n|y-z|^2}{t}}dy\nonumber =c n^d \cdot n^{-d/2}\cdot n^{-d/2}\le c,\nonumber
\end{align}
where the integration in $t$ and $s$ converges because $Hd<1$.

Similarly, again by Theorem \ref{density bivariate}
\begin{align}\label{103521}
&\,\IE\|\mu_n\|_\gamma^2=2\int_{t=\epsilon}^1\int_{s=\epsilon}^1(2\pi n)^d \exp\left[-\frac{n\left(|X_t-x|^2+|X_s-x|^2\right)}{2}\right]\frac{1}{|t-s|^\gamma}dsdt \nonumber
\\
&\le 2\int_{t=\epsilon}^1dt\int_{s=\epsilon}^1ds\int_{z\in \IR^d}\int_{y\in \IR^d}(2\pi n)^d \exp\left[-\frac{n\left(|x-y|^2+|x-z|^2\right)}{2}\right]\nonumber\\
&\quad\quad\quad\quad\quad\times \frac{1}{|t-s|^\gamma}\frac{1}{|t-s|^{Hd}}\exp\left\{-\lambda\frac{|y-z|^{2\gamma}}{C|t-s|^{2\gamma^2}} \right\}dydz \nonumber
\\
&\le c\int_{t=\epsilon}^1\int_{s=\epsilon}^1\frac{1}{|t-s|^{\gamma+Hd}}dsdt, \nonumber
\end{align}
which converges since $\gamma=1-(1+\delta)Hd$. This proves the third inequality of \eqref{103410}. The proof to the upper bound is thus finished.
\end{proof}

The following theorem addressing the Hausdorff dimension of the level sets of processes driven by fBM is an immediate consequence of the combination of Lemma \ref{124935} and Lemma \ref{124936}.
\begin{theorem}
Let $X=\{X(t), t\in \IR\}$ be a process driven by fBM on $t\in [\epsilon, 1]$.
\\
(i) If $dH >1$, then for every $x\in \IR^d$, $L_x=\varnothing$ a.s.
\\
(ii) If $dH <1$, then for every $x\in \IR^d$, 
\begin{align*}
\dim_{\mathcal{H}}L_x&=\dim_{\mathcal{P}}L_x=1-dH
\end{align*}
holds with positive probability. 
\end{theorem}

\end{document}